\documentclass[12pt]{amsart}

\usepackage[english]{babel}     
\usepackage[utf8]{inputenc}    
\usepackage[sort]{cite}           
\usepackage{fancyhdr}          
\usepackage[a4paper, left=1.2in]{geometry} 
\usepackage{xspace}            
\usepackage{dsfont}
\usepackage{amsxtra}
\usepackage{amsmath}
\usepackage{amscd}
\usepackage{amssymb}
\usepackage{amsfonts}
\usepackage[all]{xy}
\usepackage{mathrsfs}
\usepackage{amsthm} 
\usepackage{color}
\usepackage{upgreek}
\usepackage{enumerate}
\usepackage{latexsym}
\usepackage{todonotes}
\usepackage{stmaryrd}
\usepackage{nicefrac}
\usepackage{euscript}
\usepackage{mathabx}
\usepackage{enumitem}
\usepackage[foot]{amsaddr}
\usepackage{mathbbol}
\usepackage{hyperref}

\newcommand{\arxiv}[1]{\href{http://arxiv.org/abs/#1}{\tt
    arXiv:\nolinkurl{#1}}}

\newtheorem{thm}{\bf{Theorem}}[section]
\newtheorem{lemma}[thm]{Lemma}
\newtheorem{cor}[thm]{Corollary}
\newtheorem{prop}[thm]{Proposition}

\newtheorem{conj}[thm]{Conjecture}
\newtheorem{df}[thm]{Definition}
\newtheorem*{theorem-A}{Theorem A}
\newtheorem*{theorem-B}{Theorem B}
\newtheorem*{theorem-C}{Theorem C}
\newtheorem*{theorem-D}{Theorem D}
\newtheorem*{conjecture-A}{Conjecture A}
\newtheorem*{conjecture-B}{Conjecture B}
\newtheorem*{conjecture-C}{Conjecture C}
\newtheorem*{conjecture-D}{Conjecture D}

\theoremstyle{remark} 
\newtheorem{rk}[thm]{Remark}
\newtheorem{eg}[thm]{Example}

\def\bbB{\mathbb{B}}
\def\bbC{\mathbb{C}}

\def\bbG{\mathbb{G}}

\def\bbN{\mathbb{N}}

\def\bbP{\mathbb{P}}
\def\bbQ{\mathbb{Q}}
\def\bbR{\mathbb{R}}
\def\bbS{\mathbb{S}}
\def\bbT{\mathbb{T}}

\def\bbV{\mathbb{V}}

\def\bbZ{\mathbb{Z}}

\def\frakF{\mathfrak{F}}
\def\frakg{\mathfrak{G}}

\def\calA{\mathcal{A}}

\def\calE{\mathcal{E}}

\def\calM{\mathcal{M}}
\def\calN{\mathcal{N}}
\def\calO{\mathcal{O}}
\def\calP{\mathcal{P}}

\def\calS{\mathcal{S}}
\def\calT{\mathcal{T}}

\def\calW{\mathcal{W}}

\def\frake{\mathfrak{e}}

\def\frakg{\mathfrak{g}}
\def\frakh{\mathfrak{h}}

\def\frakl{\mathfrak{l}}

\def\bfe{\mathbf{e}}

\def\bfG{\mathbf{G}}
\def\bfH{\mathbf{H}}
\def\bfI{\mathbf{I}}

\def\bfS{\mathbf{S}}
\def\bfT{\mathbf{T}}

\def\bfV{\mathbf{V}}

\def\simto{\overset{\sim}\to}

\def\v{{\!\vee}}

\def\lam{\lambda}

\def\eW{\tilde{W}}

\def\mod{\operatorname{-mod}\nolimits}

\def\Ad{{{\operatorname{Ad}\nolimits}}}
\def\-{{\operatorname{-}\!}}

\def\Im{\operatorname{Im}\nolimits}

\def\End{\operatorname{End}\nolimits}

\def\Aut{\operatorname{Aut}\nolimits}

\def\Ind{\operatorname{Ind}\nolimits}

\def\Tr{\operatorname{Tr}\nolimits}

\def\Ind{\operatorname{Ind}\nolimits}

\newcommand{\Fl}{{{\mathcal F}\ell}}

\def\Tr{\operatorname{Tr}\nolimits}

\def\ch{\operatorname{ch}\nolimits}

\def\Adm{\operatorname{Adm}\nolimits}
\def\Lie{\operatorname{Lie}\nolimits}
\def\rad{\operatorname{rad}\nolimits}

\def\rot{\operatorname{rot}\nolimits}

\def\Irr{\operatorname{Irr}\nolimits}

\def\SL{\operatorname{SL}\nolimits}

\def\sl{\mathfrak{sl}}
\def\so{\mathfrak{so}}
\def\sp{\mathfrak{sp}}

\def\av{\alpha^{\!\v}}

\numberwithin{itemcounter}{subsection}
\numberwithin{equation}{section}
\title[Modularity and DAHA]{Modularity for $\mathcal{W}$-algebras and affine Springer fibres}

\author{Peng Shan$^{1,2}$}
\author{Dan Xie$^2$}
\author{Wenbin Yan$^1$}

\address{\scriptsize{$^1$} Yau Mathematics Science center, Tsinghua University, Beijing, 100084, China}
\address{\scriptsize{$^2$} Department of Mathematics, Tsinghua University, Beijing, 100084, China}

\begin{document}
\maketitle

\begin{abstract}
We construct a bijection between admissible representations for an affine Lie algebra $\frakg$ at boundary admissible levels and $\bbC^\times$ fixed points in homogeneous elliptic affine Springer fibres for the Langlands dual affine Lie algebra $\frakg^\v$. Using this bijection, we relate the modularity of the characters of admissible representations to Cherednik's Verlinde algebra construction coming from double affine Hecke algebras. Finally, we show that the expected behaviors of simple modules under quantized Drinfeld-Sokolov reductions are compatible with the reductions from affine Springer fibres to affine Spaltenstein varieties. This yields (modulo some conjectures) a similar bijection for irreducible representations of $\calW$-algebras, as well as an interpretation for their modularity properties.
\end{abstract}

\section*{Introduction}

Motivated by the mirror symmetry for 4d $\calN=2$ superconformal field theories compactified on a circle with finite size, we study the relationship between admissible representations for any affine Lie algebra $\frakg$ and $\bbC^\times$-fixed points in the elliptic affine Springer fibres for the Langlands dual affine Lie algebra $\frakg^\v$. 
For physics motivations and background of these results, we refer the readers to \cite{SXY} for details.

Admissible representations were first studied by Kac-Wakimoto \cite{KW89} in view of studying representations of affine Lie algebras whose characters admit modularity properties. These representations are equivalent to representations of the simple quotient $L_k$ of the universal affine vertex algebra of $\frakg$ at the corresponding level by \cite{A1}.

Under the $4d/2d$ duality proposed in \cite{BLLPRvR}, the vertex algebras $L_k$ as well as the $\calW$-algebras obtained from their quantized Drinfeld-Sokolov reductions appear as the image of some $4d$ theories $\calT_k$. Their associated varieties are conjectured to be the Higgs branches of such theories.
On the other hand side, the Coulomb branch of a (generalized) 4d Class $\calS$ theory is related to the moduli space of Higgs bundles on a Riemann surface \cite{GMN}. For those theory $\calT_k$'s discussed above, the Coulomb branches are related to the moduli spaces of wild Higgs bundles on $\bbP^1$ with a regular singularity at zero and an irregular singularity at infinity. 

A bijection between $\bbC^\times$-fixed points in the moduli space of Higgs bundles and admissible representations for principal $\calW$-algebras in type $A$ was first established in \cite{FN}. In this paper, we establish a similar bijection for arbitrary affine Lie algebra $\frakg$, as well as for their $\calW$-algebras, using elliptic affine Springer fibres. The relationship to Hitchin fibres is explained in Remark \ref{rk:BBAMY}. This bijection provides a new connection between admissible representations and the geometry of affine Springer fibres. We explore the compatibilities of modular group actions through this connection and some other applications.

Let us now explain the main results of the paper. 

For an affine Kac-Moody algebra $\frakg$, we consider its representations at a \emph{boundary principal admissible} level $k$, which is of the form
\[k+h^\v=\frac{h^\v}{u}.\]
Here $h^\v$ is the dual Coxeter number and $u$ is a positive integer coprime both to $h^\v$ and the dual tiers number $r^\v$. Simple admissible highest weight modules are classified by Kac-Wakimoto. Their highest weights are given by
$$\Adm_k=\{x.(k \varpi_0)\mid x\in\eW,\  x(\Pi^\v_u)\subset \Phi^\v_+\},$$
where $\Pi^\v_u=\{(u-1)c+\av_0,\ \av_i,\ \forall i\in \bar I\}$, with $\av_i$ being simple affine roots, $c\in\frakg$ is the central element, $\eW$ is the extended affine Weyl group and $\varpi_0$ is the zeroth fundamental weight.

On the other hand, let $\gamma$ be a homogeneous elliptic element of slope $\frac{u}{h^\v}$ in the Langlands dual Lie algebra $\frakg^\v$, see \eqref{gamma} for a precise definition. The affine Springer fibre $\Fl_\gamma$ inside the affine flag manifold $\Fl$ carries a $\bbC^\times$-action. Let $\Fl_\gamma^{\circ}$ be the neutral connected component.
The fixed points $\Fl^{\circ, \bbC^\times}_\gamma$ are isolated, and can be identified naturally with a subset in $\eW$. 
Then we have a bijection between the fixed points and admissible representations as follows.
\begin{theorem-A}[Theorem \ref{thm:adm}]
The map below is a bijection
\[\Fl_\gamma^{\circ,\bbC^\times}\simeq \Adm_k,\quad w\mapsto w^{-1}.(k \varpi_0).\]
\end{theorem-A}

Now, assume $\frakg$ is nontwisted.
By the work of Kac-Wakimoto \cite{KW89}, the characters for admissible representations admit modular invariance. 
More precisely, there is an action of the modular group $SL(2,\bbZ)$ on the vector space $\bbV$ spanned by the characters of admissible representations. Explicit formulae are given for the matrix coefficients of the $\bbS$ and $\bbT$ matrices.

On the other hand, by Vasserot's geometric realisations of double affine Hecke algebra(=DAHA) $\bfH_{q,t}$ via equivariant $K$-theory of affine Steinberg variety \cite{V}, for $t=q^\kappa$ with $\kappa=-\frac{u}{h^\v}$, the localised equivariant $K$-theory of the affine Springer fibre $\Fl_\gamma^{\bbC^\times}$ provides a finite dimensional module $\bfV$ for DAHA. This particular finite dimensional representation was first studied by Cherednik \cite{Ch}. It has many structures analogous to those of a Verlinde algebra. In particular, Cherednik defined projective actions of the group $PSL^c(2,\bbZ)$ on DAHA and on $\bfV$. 
The matrices $\bfS$ and $\bfT$ can also be described explicitly in terms of specialisations of Macdonald polynomials. We will show that for $q=\exp(2\pi i/\kappa)$, this yields in fact an $SL(2,\bbZ)$-action on $\bfV$ and the following result holds.

\begin{theorem-B}[Theorem \ref{thm:modularity} and Corollary \ref{corSL2}]
For $q=\exp(2\pi i/\kappa)$, there is an isomorphism of vector spaces $\bbV\simto\bfV$, induced by the bijection in Theorem A, which intertwines the $SL(2,\bbZ)$-actions on both sides.
\end{theorem-B}

Finally, we study the compatibility of the results above with the quantized Drinfeld-Sokolov reductions, and obtain the corresponding results for $\calW$-algebras. Let $f$ be a nilpotent element in the finite part $\bar{\frakg}$ of $\frakg$. We assume that it is regular in a Levi subalgebra $\bar\frakl$, and admits a good even grading. The $\calW$-algebra $\calW_k(\bar{\frakg},f)$ is obtained from the quantized Drinfeld-Sokolov reduction of the simple vertex algebra $L_k$. There are conjectural classification of simple highest weight $\calW_k(\bar{\frakg},f)$-modules in terms of the minus-Drinfeld-Sokolov reduction $\Psi_f^-$ of admissible representations, 
which has been proven for principal nilpotent $f$ and some other cases, 
see Conjecture \ref{conj} and the comments after it for more details. This conjecture implies that the set of simple highest weight $\calW_k(\bar{\frakg},f)$-modules $\Irr(\calW_k(\bar{\frakg},f))$ are parametrised by
\[\Adm_k^f=\{ x.(k \varpi_0)\mid x(\Pi^\v_u)\subset \Phi^\v_+\backslash \bar{\Phi}^\v_f\}/\sim_f,\]
where $\bar{\Phi}^\v_f$ are the coroots for $\bar\frakl$. Let $W_f$ be the Weyl group for $\bar\frakl$, the equivalence relation $\sim_f$ is given by $x\sim_f x'$ if $x'\in W_f x$.

On the other hand, let $P^\v_f$ be the parahoric subgroup in the affine Kac-Moody group $G^\v$, whose Levi subgroup has the Lie algebra $\bar\frakl$. We consider the affine Spaltenstein variety $\Fl_\gamma^f$ consisting of elements $g P^\v_f$ such that $\Ad_{g^{-1}}(\gamma)$ belongs to the Lie algebra of the radical of $P^\v_f$. 
Let $\Fl_\gamma^{f,\circ}$ be its neutral component.

\begin{theorem-C}[Theorem \ref{thm:Wbij}]
Assuming Conjecture \ref{conj} holds, then the bijection in Theorem A induces a bijection 
\[(\Fl_\gamma^{f,\circ})^{\bbC^\times}\simeq \Irr(\calW_k(\bar{\frakg},f)).
\]
\end{theorem-C}

As an application, we deduce a formula for the number of simple highest weight $\calW_k(\bar{\frakg},f)$-modules (based on the conjecture), see Corollary \ref{cor:count}.
It is obtained by counting the fixed points in $(\Fl_\gamma^{f,\circ})^{\bbC^\times}$. 
This formula gives a new way to detect easily when $\calW_k(\bar{\frakg},f)$ is zero, or equal to $\bbC$, see Remarks \ref{rk1}, \ref{rk2}.

Finally, in terms of DAHA representations, passing from affine Springer fibre to affine Spaltenstein varieties corresponds to projecting the representation $\bfV$ to $\bfV_f=\bfe_f\bfV$, where $\bfe_f$ is the sign idempotent in the group algebra of $W_f$. We further show that this is compatible with the map from the space of characters $\bbV$ to the space $\bbV_f$ spanned by the characters of $\calW_k(\bar{\frakg},f)$-representations, which is given by the Drinfeld-Sokolov reduction. As a result, we obtain

\begin{theorem-D}[Theorem \ref{thm:modW}]
The isomorphism in Theorem B induces an isomorphism 
\[\bbV_f\simto \bfV_f,\]
which again intertwines the $\SL(2,\bbZ)$-actions on both sides.
\end{theorem-D}

\medskip

\subsection*{Acknowledgements} We would like to thank Tomoyuki Arakawa,  Ivan Cherednik, Gurbir Dhillon, Pavel Etingof and Thomas Haines for helpful discussions.
This work is finished when PS is participating the program ``Arithmetic Quantum Field Theory'' at the Center of Mathematical Sciences and Applications at Harvard University. She would like to thank the organisers and the institute for excellent working environment.

PS is supported by NSFC Grant No.~12225108.
DX is supported by national key research
and development program of China (NO. 2020YFA0713000).
WY is supported by national key research
and development program of China (NO. 2022ZD0117000).
 DX and WY are supported by NSFC with
Grant NO. 12247103.

\section{Admissible modules and fixed points}

\subsection{Notation}\label{ss:notation}

Let $A=(a_{ij})_{i,j\in I}$ be a generalized Cartan matrix of affine type $\tilde{X}_N^{(r)}$ as listed in Table Aff $r$ of \cite[Section~4.8]{Kac}. Write $I=\{0,1,...,\ell\}$. The number $r$ is called the \emph{tier} number. Let $(a_0,...,a_\ell)$, resp. $(a^\v_0,...,a^\v_\ell)$, be the unique vector of relatively prime positive integers such that
$$(a_0,...,a_\ell) A^t=0, \  \text{resp. } (a^\v_0,...,a^\v_\ell)A=0,$$
where $A^t$ is the transposed matrix.
Recall that $a_0=1$ if $A\neq A^{(2)}_{2\ell}$,  and $a_0=2$ if $A= A^{(2)}_{2\ell}$, while $a^\v_0=1$ in all cases.
The number $h^\v=\sum_{i\in I}a^\v_i$ is the dual Coxeter number of $A$.

Let $\frakg=\frakg(A)$ be the complex affine Lie algebra associated with $A$.
Fix a Cartan subalgebra $\frakh$. It has the form $\frakh=\frakh'\oplus\bbC d$, where $\frakh'$ is an $(\ell+1)$-dimensional vector spaces with a basis given by the set of simple coroots $\Pi^\v=\{\av_0,...,\av_\ell\}$. Let $\Pi=\{\alpha_0,...,\alpha_\ell\}$ be the set of simple roots given by $a_i\alpha_i=a_i^\v\av_i$. The space $\frakh'$ carries a symmetric bilinear form $(\ ,\ )$ such that
$$(\av_i, \av_j)=a_{ij}a_j/a^\v_j\quad \forall i,j\in I.$$ 
The kernel of this form is spanned by the canonical central element
$$c=\sum_{i\in I}a^\v_i\av_i.$$
It extends to a unique nondegenerate symmetric bilinear form on $\frakh$ such that $(d,c)=1$, $(d,d)=0$ and $(d,\alpha^\v_i)=0$ for $i=1,...,\ell$. We use it to identify $\frakh\simeq\frakh^\ast$. The natural pairing $\frakh^\ast\times\frakh\to\bbC$ will be denoted as $\langle\ ,\ \rangle$.

The affine Weyl group $W$ is the subgroup of $\Aut(\frakh')$ generated by the simple reflections $s_i$ given by $s_i(v)=v- (\alpha_i,v)\av_i$ for $v\in\frakh'$ and $i\in I$. The set $\Phi^\v=W(\Pi^\v)$ is the set of \emph{real coroots}. Let $\Phi^\v_+$, resp. $\Phi^\v_-$, be the subsets of positive, resp. negative, real coroots.

Let $\bar{\frakg}$ be the Lie algebra associated with the finite Cartan matrix $A_0=(a_{ij})_{i,j\in \bar I}$, where $\bar I=\{1,...,\ell\}$.
Let $\bar{\frakh}$ be the Cartan subalgebra spanned by $\av_i$ for $i\in \bar I$. Let $\bar \Phi^\v$ be the set of finite coroots, and $\bar{W}=\langle s_i\mid i\in \bar I\rangle$ be the finite Weyl group. Consider the root lattice $ Q=\sum_{i\in \bar I} \bbZ\alpha_i$ and the coroot lattice $ Q^\v=\sum_{i\in \bar I} \bbZ\av_i$.
If $r=1$, then $a_i^\v$ always divides $a_i$, so $ Q\supset Q^\v$.
If $r>1$, then $a_i$ always divides $a_i^\v$ for $i\in \bar I$, so $ Q\subset Q^\v$.

Let $r^\v$ be the tier number of $A^t$. Set
$$M=\begin{cases}  Q, &\text{ if }r^\v=1, \\  Q^\v, &\text{ if }r^\v>1.\end{cases}$$
The affine Weyl group $W$ is a semi-direct product of $\bar{W}$ with $M$. More precisely, 
consider the map $t:\frakh'\to \Aut(\frakh')$ given by $b\mapsto t_b$ with
\begin{equation*}
t_b(h)=h+(h,c)b-\left((h,b)+\frac{1}{2}(h,c)(b,b)\right)c,\quad\forall\ h\in\frakh'.
\end{equation*}
For any subset $N\in \frakh'$, write $t_N=\{t_b\mid b\in N\}$. 
Then $W=\bar{W}\ltimes t_M$.

The extended affine Weyl group $\eW=\bar{W}\ltimes t_{\tilde M}$, where
$\tilde{M}$ is the dual lattice for $ Q+ Q^\v$. Recall that for 
any lattice $L\subset \frakh'$, its dual lattice is $L^\ast=\{x\in\bbQ L\mid (x,L)\subset\bbZ\}$.
Since $( Q,  Q^\v)\subset \bbZ$, we have 
$$M\subset \tilde M=( Q+ Q^\v)^\ast \subset M^\ast$$
as subsets in $\frakh$.
If $r=1$, then $ Q\supset  Q^\v$. In this case $\tilde M$ is the coweight lattice $P^\v$. If $r>1$, then $ Q\subset  Q^\v$, and $\tilde M$ is the weight lattice $P$. For $w\in\eW$, let 
$$\Phi^\v(w)=\Phi^\v_+\cap w^{-1}(\Phi^\v_-).$$ 
Its cardinality is equal to the length $\ell(w)$ of $w$. 
Set $\epsilon(w)=(-1)^{l(w)}$.
Let $\Omega=\{w\in\tilde W\mid w(\Pi^\v)=\Pi^\v\}$.

The fundamental weights $ \varpi_i$ are given by $\langle  \varpi_i,\av_j\rangle=\delta_{ij}$ for $i,j\in I$. Together with the imaginary root $\delta=\sum_{i\in I} a_i\alpha_i$, they form a basis of ${\frakh}^\ast$.
The dual action of $\eW$ on ${\frakh}^\ast$ is given by
$$t_b(\lambda)=\lambda+\lambda(c) b^\ast-\left((\lambda, b^\ast)+\frac{1}{2}\lambda(c)(b^\ast, b^\ast)\right)\delta,\quad\forall b\in\tilde M,\ \lambda\in\frakh^\ast.$$
Here $b^\ast$ is the image of $b$ under the isomorphism $\frakh\simeq \frakh^\ast$.
Let $\rho=\sum_{i\in I} \varpi_i$. Then $\langle \rho, c\rangle=h^\v$.
The dot action for an element $w$ in $\eW$ on $\frakh^\ast$ is $w.\lambda=w(\lambda+\rho)-\rho$ for $\lambda\in \frakh^\ast$.
Let $\bar \varpi_i$, $\bar\rho$ be respectively the image of $ \varpi_i$ and $\rho$ under the natural projection $(\frakh')^\ast\to\bar{\frakh}^\ast$ for $i\in \bar I$.

\medskip

\subsection{Principal admissible weights}

For $\lam\in \frakh^\ast$, let
$$\Phi^\v_\lambda=\{\av\in \Phi^\v\mid \langle\lambda+\rho, \av\rangle\in\bbZ\}$$
be the set of integral real coroots with respect to $\lambda$. 
\begin{df} \cite{KW89}
A weight $\lambda$ is called \emph{principal admissible} if
\begin{enumerate}
\item $\lambda$ is regular dominant, i.e., $\langle\lambda+\rho, \av\rangle \not\in\{0,-1,-2,...\}$ for all $\av\in \Phi^\v_+$.
\item there exists a linear isomorphism $\varphi: \frakh\to\frakh$ such that 
$\varphi(\Phi^\v_\lambda)=\Phi^\v$.
\end{enumerate}
\end{df}
Such weights are classified by Kac-Wakimoto \cite[Theorem 2.1]{KW89}.
Their level has the following form
\begin{equation*}
k+h^\v=\frac{p}{u},
\end{equation*}
where $u$ and $p$ are positive integers such that $(p,u)=(r^\v, u)=1$ and $p\geqslant h^\v$.
In this paper, we will focus on the \emph{boundary principal admissible level} 
which refers to the case
$p=h^\v$. So from now on, we fix
\begin{equation}\label{df:k}
k+h^\v=\frac{h^\v}{u}\ ,\text{ with }(h^\v,u)=(r^\v, u)=1.
\end{equation}

By loc.~cit (see also \cite[Proposition~2]{KW16}), the set of boundary principal admissible weights consists of the following elements
$$\Adm_k=\{x.(k \varpi_0)\mid x\in\eW,\  x(\Pi^\v_u)\subset \Phi^\v_+\},$$
where $\Pi^\v_u=\{(u-1)c+\av_0,\ \av_i,\ \forall i\in \bar I\}$.
Note that $x.(k \varpi_0)=x'.(k \varpi_0)$ if and only if $x^{-1}x'(\Pi^\v_u)=\Pi^\v_u$.
Let
\begin{align*}
\eW_u =\{x\in \eW\mid x(\Pi^\v_u)\subset \Phi^\v_+\},\quad
\Omega_u=\{x\in \eW\mid x(\Pi^\v_u)=\Pi^\v_u\}.
\end{align*}
Then $\Omega_u$ acts on $\eW_u$ from the right, and we have a bijection
\begin{equation}\label{eq:b0}
\eW_u/\Omega_u\simto \Adm_k,\quad x\mapsto x.(k \varpi_0).
\end{equation}

\subsection{Non-twisted case}
Assume in this subsection that the affine Lie algebra $\frakg$ is non-twisted. In other words $r=1$.
In this case $M= Q^\v$, and $\tilde{M}= P^\v$.
The set $\eW_u/\Omega_u$ has a more explicit description as follows.

For any $b\in \tilde{M}$, the element $t_b\in\eW$ admits a unique decomposition
\[t_b=\pi_bu_b \quad \text{ such that }u_b\in \bar{W} \text{ and } \Phi^\v(\pi_b)\cap \bar \Phi^\v=\emptyset.\]
Moreover, $u_b$ is the unique element in $\bar{W}$ such that
\begin{itemize}
\item $b_-=u_b(b)$ belongs to $\tilde{M}_-$,
\item for any $\alpha^\v\in \Phi^\v(u_b)\cap\bar \Phi^\v$, we have $(\alpha^\v,b)\neq 0$. 
\end{itemize} 
See e.g. \cite[Prop.~3.1.2]{Ch}. Here $\tilde{M}_-=\{\lambda\in\tilde{M}\mid (\lambda,\alpha_i^\v)\leqslant 0,\ \forall\ i\in\bar I\}.$

\smallskip

Let $J=\{i\in \bar I\mid a_i=1\}$. The set $\Omega$ admits the following description 
\[\Omega=\{\pi_j:=\pi_{\bar \varpi_j}\mid j\in J\}.\] 
Note that for $j\in J$, the weight $\bar \varpi_j$ is also the $j$-th fundamental coweight. 
In particular, it belongs to $\tilde{M}$.
We have an isomorphism of groups
$$\Omega\simeq \tilde M/M,\quad \pi_j\mapsto \bar \varpi_j+M.$$
Similarly, $\Omega_u=\{\pi_{u\bar \varpi_j}\mid j\in J\}\simeq u\tilde M/uM$.

Now, consider the following set
\begin{equation}\label{eq:sigmau}
\Sigma_u=\left\{b\in \tilde{M} \Bigm\vert \begin{array}{l}\text{either } u+(\theta^\v,b_-)>0,\\ \text{or }u+(\theta^\v,b_-)=0\text{ and }u_b^{-1}(\theta^\v)\in \bar{\Phi}^\v_-\end{array}\right\}.
\end{equation}

\begin{lemma} \label{lem:sigmau}
\begin{itemize}
\item[(a)] We have $\eW_u=\{\pi_b \mid b\in \Sigma_u\}$.

\item[(b)] The action of $\Omega_u$ on $\eW_u$ is given by $\pi_b\pi_{u \bar{ \varpi}_j}=\pi_{b+u_b^{-1}(u\bar{ \varpi}_j)}$.
\end{itemize}
\end{lemma}
\begin{proof}
For $b\in \tilde{M}$ and $w\in\bar{W}$, the condition $t_bw(\Pi^\v_u)\subset \Phi^\v_+$ is equivalent to
\begin{align}\label{eq:desc}
\begin{split}
t_bw(\alpha^\v_i)&=w(\alpha^\v_i)-(w(\alpha^\v_i),b)c\in \Phi^\v_+, \ \forall i,\\
t_bw(uc-\theta^\v)&=-w(\theta^\v)+[u+(w(\theta^\v),b)]c\in\Phi^\v_+.
\end{split}
\end{align}
The first line implies that $-(\alpha^\v_i, w^{-1}b)\geqslant 0$ for all $i$, hence $w^{-1}b\in \tilde{M}_-$.
Further, if $(\alpha^\v_i, w^{-1}b)=0$, then $w(\alpha^\v_i)\in\bar{\Phi}^\v_+$. Thus $w^{-1}=u_b$ and $t_bw=t_bu_b^{-1}=\pi_b$.

Now, the condition $t_bw(uc-\theta^\v)\in \Phi^\v_+$ becomes $u+(\theta^\v,b_-)>0$ or $u+(\theta^\v,b_-)=0$ and $u_b^{-1}(\theta^\v)\in \bar{\Phi}^\v_-$. Part (a) follows.

The equality in part (b) follows from the definition.
\end{proof}

We impose the equivalence relation on $\Sigma_u$ given by $b\sim b'$ if $\pi_b$ and $\pi_{b'}$ belong to the same $\Omega_u$-orbit. Combining the previous lemma and \eqref{eq:b0}, we get
\begin{equation}
\Sigma_u/\!\sim\ \simto \Adm_k,\quad b\mapsto \pi_b.(k \varpi_0).
\end{equation}

\smallskip

\subsection{The affine flag manifold}
Now, back to the general situation, we recall some basic facts about the affine flag manifold.


Let $\frakg^\v$ be the Langlands dual of $\frakg$, which is the affine Kac-Moody algebra whose generalized Cartan matrix is  $A^t$.
Let $G^\v$ be the Kac-Moody group of adjoint type whose Lie algebra is $\frakg^\v$.
It can be realised as $G^\v=G^{\v}_c\rtimes \bbG_m^{\rot}$, where $G^{\v}_c$ is a central extension of a possibly twisted loop group $L\bar{G}^\v$ constructed as follows.

According to Table Aff $r$ of \cite[Section~4.8]{Kac}, the Lie algebra $\frakg^\v$ is of type $\tilde{Y}_N^{(r^\v)}$, and there is an automorphism $\theta$ of the Dynkin diagram of $\tilde{Y}_N$ of order $r^\v$ attached to it. Let $\bbG$ be the connected quasi-simple adjoint algebraic group of type $Y_N$. Let $\mathbb{g}$ be its Lie algebra.
We fix a pinning $(\bbT, \bbB,...)$ for $\bbG$. The automorphism $\theta$ lifts to a pinned automorphism of $\bbG$, which we denote again by $\theta$. The subgroup $\bbG^\theta$ of $\theta$-fixed points is connected. Its Lie algebra is the \emph{Langlands dual} of $\bar\frakg$. The fixed points $\bbT^\theta$ is a maximal torus in $\bbG^\theta$. The root system for $(\bbG^\theta,\bbT^{\theta})$ is Langlands dual to the one for $(\bar\frakg,\bar\frakh)$. The relative Weyl group $N_{\bbG^\theta}(\bbT^{\theta})/\bbT^{\theta}$ is $\bar W$.
Let $\phi$ be the set of absolute roots for $(\bbG,\bbT)$.
Then $\bar\Phi^\v$ identifies with the average of the $\theta$-orbits on $\phi$.

Let $F=\bbC(\!(z)\!)$ be the field of formal Laurent series with coefficient in $\bbC$, and $\calO_F=\mathbb{C}[[z]]$ be the subring of formal power series.
Consider the Galois extension $F_{r^\v}=\bbC(\!(z^{1/r^\v})\!)$ of $F$ with Galois group $\mu_{r^\v}$. Pick $\zeta$ a generator of $\mu_{r^\v}$. Then it induces an automorphism of $\bbG(F_{r^\v})$ denoted again by $\zeta$. 
The homomorphism $\theta$ gives a descent datum of $\bbG\otimes F_{r^\v}$ from $F_{r^\v}$ to $F$. Let $\bar G^\v$ be the result group scheme over $F$. Explicitly, for any $F$-algebra $R$, we have
$$\bar G^\v(R)=\{g\in \bbG(R)\otimes_FF_{r^\v}\mid \zeta(g)=\theta(g)\}.$$
Similar construction with $\bbG$ replaced by $\bbT$ defines a maximal torus $\bar T^\v$ of $\bar G^\v$ over $F$. We also have a special parahoric subgroup of $\bar G^\v$ defined by
$$\bar\bfG^\v(R)=\{g\in \bbG(R)\otimes_{\calO_F}\calO_{F_{r^\v}}\mid \zeta(g)=\theta(g)\},$$
for any ring $R$ over $\calO_{F}$.

The loop group $L\bar G^\v$ is the group ind-scheme over $\bbC$ whose $R$-points are given by $L\bar G^\v(R)=\bar G^\v(R(\!(z)\!))$ for any commutative $\bbC$-algebra $R$. It has a parahoric subgroup $L^+\bar G^\v$ whose $R$ points are given by $\bar\bfG^\v(R[[z]])$. The evaluation map $z\mapsto 0$ induces a group homomorphism $L^+\bar G^\v\to \bbG^\theta$. The Iwahori subgroup $\bfI^\v$ is the preimage of $\bbB^\theta$ under this map.

The loop group $L\bar G^\v$ has a central extension $1\to \bbG_m^{cen}\to G^\v_c\to L\bar G^\v\to 1$, see for example \cite[Section~2.5.2]{OY} and the reference therein. There is a $\bbG^{rot}_m$ action on $G^\v_c$ which commute with $\bbG_m^{cen}$ and acts on $L\bar G^\v$ by $tg(z^{1/r^\v})=g(tz^{1/r^\v})$.
The Kac-Moody group $G^\v=G^\v_c\rtimes \bbG^{rot}_m$. The Iwahori subgroup is $I^\v=\bbG_m^{cen}\times\bfI^\v\rtimes \bbG^{rot}_m$.

The Lie algebra of $L\bar G^\v$ is $\bar\frakg^\v(F)=\{x(z^{1/r^\v})\in\mathbb{g}\otimes F_{r^\v}\mid \theta(x(z^{1/r^\v}))=x(\zeta z^{1/r^\v})\}$. Let $\delta$, $\varpi_0$ be respectively the generators of the cocharacter lattice $X_\ast(\bbG_m^{rot})$ and $X_\ast (\bbG_m^{cen})$.
Then $\frakg^\v=\bar\frakg^\v(F)\oplus\bbC\delta\oplus\bbC\varpi_0$, and $\frakg^{\v,\ast}=\bar\frakg(F)^\ast\oplus\bbC c\oplus\bbC d$.
The Lie algebra of $I^\v$ is the subspace spanned by the root space for roots in $\Phi^\v_+$ and those in $\mathbb{N}\delta$.

The affine flag manifold is the fppf quotient $\Fl=G^\v/I^\v$. It is an ind-scheme whose connected components are naturally parametrised by $\Omega$.
The torus $\bbT^\theta$ acts on $\Fl$ via left multiplication. The fixed points set is naturally in bijection with the extended affine group $\bar W\ltimes X_\ast(\bbT)_{\mu_{r^\v}}$ by \cite[Prop.~13]{HR}. This coincides with the extended affine Weyl group $\eW$ considered in Section \ref{ss:notation}. In other words, we claim there is a natural identification $X_\ast(\bbT)_{\mu_{r^\v}}=\tilde M$.

To view this, we may first replace $\bbG$ by its simply connected cover $\bbG_{sc}$, denote the corresponding maximal torus by $\bbT_{sc}$. Then its enough to show $X_\ast(\bbT_{sc})_{\mu_{r^\v}}=M$.
The adjoint case follows, because $X_\ast(\bbT_{sc})_{\mu_{r^\v}}$ and $X_\ast(\bbT)_{\mu_{r^\v}}$ differ by $\Omega$.
Now, if $r^\v=1$, then $X_\ast(\bbT_{sc})$ is the coroot lattice for $\bar{\frakg}^\v$, which is the root lattice $Q$ for $\bar{\frakg}$. So it equals $M$ in this case. If $r^\v>1$, then $X_\ast(\bbT_{sc})_{\mu_{r^\v}}$ is the lattice 
generated by the echlonnage coroots by \cite[Lemma 15]{HR}. 
By \cite[Cor.~5.3]{H}, these echlonnage coroots identifies with the roots for $\theta$-fixed subgroup in the Langlands dual of $\bbG_{sc}$. But since $\bbG_{sc}$ is simply-laced, they can also be identified with the roots for $\bbG^\theta_{sc}$. The Lie algebra of $\bbG^\theta_{sc}$ is Langlands dual to $\bar\frakg$. So putting these together, 
the echlonnage coroots are identified with the coroots for $\bar\frakg$.
Thus $X_\ast(\bbT)_{\mu_{r^\v}}$ coincides with the coroot lattice $Q^\v$ for $\bar\frakg$, which is $M$ for $r^\v>1$.

\subsection{The affine Springer fibre}
Now, we can discuss the relationships between admissible weights and fixed points in affine Springer fibres.

For each $\alpha^\v\in \Phi^\v$, pick a nonzero root vector $e_{\alpha^\v}$ in $\frakg^\v_{\alpha^\v}$. Consider the element 
\begin{equation}\label{gamma}
\gamma=e_{(u-1)c+\av_0}+\sum_{i=1}^\ell e_{\av_i}\in \frakg^\v.
\end{equation}
Here $u$ is the same integer as the one in \eqref{df:k}.
It is a homogeneous regular elliptic element in the sense of \cite{KL}.
The affine Springer fibre associated with $\gamma$ is the closed subscheme in $\Fl$ whose closed points are given by
$$\Fl_\gamma=\{gI^\v\mid \Ad_{g^{-1}}(\gamma)\in \mathrm{Lie}(I^\v)\},$$
here $\Ad$ denotes the adjoint representation of $G^\v$ on $\frakg^\v$.
It is a finite dimensional projective variety.

Let $\bar\rho$ be the cocharacter of $\bbT$ which is the sum of fundamental coweights for $\bbG$. Its image is fixed by $\theta$.
Consider the one dimensional subtorus
\begin{equation}\label{Cstar}
\nu: \bbC^\times \to \bbT^\theta\times \bbG_m^{\rot}\subset T^\v, \quad t\mapsto \nu(t)=(t^{u\!{\bar \rho}}, t^{h^\v}).
\end{equation}
Then, we have $\Fl^{\bbC^\times}=\Fl^{T^\v}$. 
Indeed, it is clear that $T^\v$-fixed points are $\bbC^\times$-fixed. 
Next, for any $\alpha^\v=\bar\alpha^\v+mc\in\Phi^\v$, the conjugation action of $\bbC^\times$ on $\frakg^\v_{\alpha^\v}$ sends a vector $e_{\alpha^\v}$ to $t^{u\langle \bar\alpha^\v,{\bar \rho}\rangle +mh^\v}e_{\alpha^\v}$. Since $|\langle \bar\alpha^\v,{\bar \rho}\rangle|\leqslant h^\v -1$, we have $u\langle \bar\alpha^\v,{\bar \rho}\rangle +mh^\v>0$. Hence $w$ is the only $\bbC^\times$-fixed point in the Bruhat cell $I^\v wI^\v/I^\v\subset\Fl$. We deduce that $\Fl^{\bbC^\times}=\Fl^{T^\v}$.

Next, for $t\in  \bbC^\times$, we have $\Ad_{\nu(t)^{-1}}(\gamma)=t^u\gamma$.
So $\Fl_\gamma$ is stable under this $\bbC^\times$-action. The fixed points set can be described as follows
\begin{align*}
\Fl_\gamma^{\bbC^\times}&=\Fl_\gamma\cap \Fl^{T^\v}=\{w\in \eW\mid w^{-1}(\Pi^\v_u)\subset \Phi^\v_+\}.
\end{align*}
Thus we obtain a bijection 
\begin{equation}\label{eq:b1}
\eW_u\simeq \Fl_\gamma^{\bbC^\times},\quad w\mapsto w^{-1}.
\end{equation}

\begin{thm}\label{thm:adm}
We have a bijection
\begin{equation}\label{eq:bij}
\Fl_\gamma^{\circ,\bbC^\times}\simeq \Adm_k,\quad w\mapsto w^{-1}.(k \varpi_0).
\end{equation}
\end{thm}

\begin{proof}
By \eqref{eq:b0} and \eqref{eq:b1}, the assignment $w\mapsto w^{-1}.(k \varpi_0)$ yields a well defined surjective map
\begin{equation}\label{eq:b2}
\Fl_\gamma^{\bbC^\times}\to \Adm_k.
\end{equation}
The neutral component $\Fl^{\circ}$ is isomorphic to the affine flag manifold associated with the simply-connected group $G^\v_{sc}$. In this case, the affine Springer fibre $\Fl^{\circ}_\gamma$ is connected.
The connected component of $\Fl_\gamma$ is naturally indexed by $\Omega_u$.
Each of them is contained in a unique connected component of $\Fl$. 
The bijection $\pi_0(\Fl)=\pi_0(\Fl_\gamma)$ obtained in this way matches with the obvious identification $\Omega\to\Omega_u$.
The $\Omega_u$-action on $\Fl_\gamma$ sends fixed points to fixed points, and yields a bijection
\begin{equation*}
\Omega_u\backslash \Fl_\gamma^{\bbC^\times}\simeq \Fl_\gamma^{\circ,\bbC^\times}.
\end{equation*}
\medskip
So by \eqref{eq:b0}, the map \eqref{eq:b2} is a bijection.
\end{proof}

\begin{rk}\label{rk:BBAMY}
By \cite[Theorem 2.8.1]{BBAMY} the fiber $\Fl_\gamma$ is homeomorphic to a central Lagrangian fibre of the moduli spaces $\calM_\gamma$ of wild Higgs bundles, which corresponds to the Coulomb branch predicted by physics theory. In particular, $\Fl_\gamma$ and $\calM_\gamma$ have the same $\bbC^\times$-fixed points. This gives the precise relationship between our bijection and the one expected in physics theory.
\end{rk}
\medskip

\section{Modularity}

In the rest of this paper, we assume that $\frakg$ is an untwisted affine Lie algebra, so $r=1$.

\subsection{Modular groups}

The modular group $SL(2,\bbZ)$ consists of $2\times 2$ matrices with integer coefficients and determinant one. It is generated by
\begin{equation}\label{eq:sigmatau}
\tau_+=\left(\begin{array}{cc} 1 &1\\ 0 &1\end{array}\right),\quad \sigma=\left(\begin{array}{cc} 0 &-1\\ 1 &0\end{array}\right).
\end{equation}
Let $PSL^c(2,\bbZ)$ be the universal central extension of $PSL(2,\bbZ)$. It is the group generated by $\tau_+$ and $\sigma$ subject to the relation $(\sigma\tau_+)^3=\sigma^2$. We have $SL(2,\bbZ)= PSL^c(2,\bbZ)/\langle \sigma^4-1\rangle$. 

The group $GL(2,\bbZ)=SL(2,\bbZ)\rtimes \langle \varepsilon\rangle$, where $\varepsilon=\left(\begin{array}{cc} 0 &-1\\ -1 &0\end{array}\right)$ has order two. The universal central extension $PGL^c(2,\bbZ)$ is generated by $\varepsilon$ and $\tau_+$, and subject to the Steinberg relation
$$\tau_+\tau_-^{-1}\tau_+=\tau_-^{-1}\tau_+\tau_-^{-1},$$
with $\tau_-=\varepsilon\tau_+\varepsilon$.
We have $PSL^c(2,\bbZ)\hookrightarrow PGL^c(2,\bbZ)$ such that $\sigma\mapsto \tau_+\tau_-^{-1}\tau_+$.

\medskip

\subsection{Modularity of characters}
Let 
\begin{align*}
Y&=\{\lambda\in\frakh\mid\mathrm{Re}(\lambda,c)>0\}\\
&=\{(z,\tau,t):=2\pi i(z+tc-\tau d)\mid z\in\bar\frakh,\ ,\tau,t\in\bbC,\ \Im(\tau)>0\}.
\end{align*}
The modular group $SL(2,\bbZ)$ acts on $Y$ by
\[g\cdot (z,\tau,t)=\left(\frac{z}{c\tau+d}, \frac{a\tau+b}{c\tau+d},t-\frac{c(z,z)}{2(c\tau+d)}\right),\text{ for }g=\left(\begin{array}{cc} a &b\\ c &d\end{array}\right).\]
This induces a right action of $SL(2,\bbZ)$ on the space $\calM(Y)$ of meromorphic functions on $Y$.

For $\lambda\in \frakh^\ast$ of level $k$, let $L(\lambda)$ be the irreducible $\frakg$-module of highest weight $\lambda$.
The \emph{normalized character} of $L(\lambda)$ is defined as
\begin{equation}
\ch_{\lambda}(v)=e^{2\pi i\tau s_{\lambda}}\Tr_{L(\lambda)}(e^v),\quad v\in Y,
\end{equation}
where $s_{\lambda}=\frac{|\lambda+{\bar \rho}|^2}{2(k+h^\v)}-\frac{|{\bar \rho}|^2}{2h^\v}$.
For $\lambda\in\Adm_k$, these characters are linearly independent meromorphic functions on $Y$.
Let $\bbV$ be the vector subspace they span in $\calM(Y)$.

Recall that $\Adm_k=\{\pi_b.(k \varpi_0)\mid b\in \Sigma_u/\!\sim\}$.
We abbreviate $\ch_b:=\ch_{\pi_b.(k \varpi_0)}$.
In \cite{KW89}, Kac-Wakimoto proved that $\bbV$ is invariant under the $SL_2(\bbZ)$-action. 
Let $\bbT$, $\bbS$ be respectively the matrices representing $\tau_+$, $\sigma$ on the basis $\{\ch_b\mid b\in\Sigma_u/\sim\}$. Then their matrix coefficients are given by \cite[Theorem 3.6]{KW89} 
\begin{eqnarray}\label{eq:ST}
\begin{split}
\bbS_{b,b'}&=\left|\frac{P^\v}{uh^\v  Q^\v}\right|^{-1/2}\epsilon(u_bu_{b'})\left(\prod_{\alpha\in\bar\Phi^\v_+}2\sin\frac{\pi u(\alpha,{\bar \rho})}{h^\v}\right)
e^{-2\pi i\left(\frac{h^\v}{u}(b,b')+(b+b',{\bar \rho})\right)},\\
\bbT_{b,b'}&=e^{\frac{\pi i u}{h^\v}\left(|u_b^{-1}({\bar \rho})+\frac{h^\v}{u} b|^2-\frac{|{\bar \rho}|^2}{2u}\right)}\delta_{b,b'}.
\end{split}
\end{eqnarray}

\medskip

\subsection{Double affine Hecke algebras}
Our next goal is to interpret the modularity of characters in terms of double affine Hecke algebras(=DAHA).

In this section, we consider the DAHA attached to the root system of $\frakg^\v$. Recall that its underlying finite Lie algebra is $\bar{\frakg}^\v$ with the root system $\bar \Phi^\v=\bar W(\bar\Pi^\v)$. Let $\bar\Phi^\v_s$, resp. $\bar\Phi^\v_l$, be the subsets of short, resp. long roots. Then 
$$\Phi^\v=\{\alpha+nc\mid n\in\bbZ, \alpha\in\bar\Phi^\v_{s}\}\cup \{\alpha+nr^\v c\mid n\in\bbZ, \alpha\in\bar\Phi^\v_{l}\}.$$
Since we have assumed $r=1$, we have $M= Q^\v$, $\tilde{M}= P^\v$, and they are respectively the root and weight lattices for $\bar{\frakg}^\v$.
This agrees with the setup in \cite[Section~3.1.1]{Ch}.

Let $m$ be the least natural number such that $(P^\v,P^\v)=\frac{1}{m}\bbZ$. 
We fix a coefficient ring $\bbC_{q,t}=\bbC[q^{\pm 1/m},t^{\pm 1}]$, where $q,t$ are formal variables.
Let $X_1,...,X_\ell$ be formal variables. For $b\in P^\v$ and $n\in\frac{1}{m}\bbZ$, set
\begin{equation}\label{r1}
X_{b+nc}=q^n\prod_{i=1}^\ell X_i^{n_i},\quad\text{ if }b=\sum_{i=1}^\ell n_i\bar{ \varpi}^\v_i.
\end{equation}
Recall that $\Omega=\{\pi_j\mid j\in J\}$, where $\pi_j=\pi_{\bar{ \varpi}_j}$.
The extended affine Weyl group $\eW=\Omega\ltimes W$ acts on $\bbC[X^{\pm 1}_1,...,X^{\pm 1}_\ell]$ by $w(X_{b+nc})=X_{w(b+nc)}$ for $w\in\eW$.

\begin{df}\cite[Def.~3.2.1]{Ch}
The double affine Hecke algebra $\bfH_{q,t}$ (of equal parameters) attached to the root system of $\frakg^\v$ is the unital $\bbC_{q,t}$-algebra generated by $T_0$, $T_1$,...,$T_{\ell-1}$, $X_1$,...,$X_\ell$ and the group $\Omega$ subject to the relation \eqref{r1} and the relations below:
\begin{enumerate}
\item $(T_i-t^{1/2})(T_i+t^{-1/2})=0$,\ $\forall\ 0\leqslant i\leqslant \ell-1$,
\item $T_iT_{i'}T_i...=T_{i'}T_iT_{i'}...$ with $m_{i,{i'}}$ factors on each side, where $m_{i,i'}$ is the order of $s_is_{i'}$ in the affine Weyl group,
\item $\pi_j T_i\pi_j^{-1}=T_{i'}$, if $\pi_j(\alpha_i)=\alpha_{i'}$ $\forall\ j\in J$,
\item $T_iX_bT_i=X_bX_{\alpha_i}^{-1}$, if $(b,\alpha_i^\v)=1$, $\forall\ 0\leqslant i\leqslant \ell-1$,
\item $T_iX_b=X_bT_i$ if $(b,\alpha_i^\v)=0$, $\forall\ 0\leqslant i\leqslant \ell-1$,
\item $\pi_jX_b\pi_j^{-1}=X_{\pi_j(b)}$, $\forall\ j\in J$.
\end{enumerate}
\end{df}

For $x=\pi_j w\in\tilde W$, with $w$ belonging to the affine Weyl group, pick a reduced expression $w=s_{i_1}...s_{i_n}$, set $T_x=\pi_jT_{i_1}...T_{i_n}$. This element is independent of the choice of reduced expression. Set $Y_i=T_{\bar{ \varpi}^\v_i}$ for $1\leqslant i\leqslant \ell$. Then $Y_i$ are pairwise commuting elements in $\bfH_{q,t}$. For $b=\sum_{i=1}^\ell n_i\bar{ \varpi}^\v_i$, we set $Y_b=\prod_{i=1}^\ell Y_i^{n_i}$.

Let $ \Aut(\bfH_{q,t})$ be the group of $\bbC$-linear automorphisms of $\bfH_{q,t}$.
By \cite[Section.~3.2.2]{Ch}, there is a well defined group automorphism
\begin{equation*}
PGL^c(2, \bbZ)\to \Aut(\bfH_{q,t})
\end{equation*}
which sends the generators $\varepsilon$ and $\tau_+$ of $PGL^c(2, \bbZ)$ to the following automorphisms on $\bfH_{q,t}$:
\begin{align*}
&\varepsilon: X_i\mapsto Y_i,\ Y_i\mapsto X_i,\ T_i\mapsto T_i^{-1},\ \pi_j\mapsto X_jT_{\pi_j^{-1}},\  t\mapsto t^{-1},\ q\mapsto q^{-1},\\
&\tau_+: X_b\mapsto X_b,\ T_0\mapsto X_0^{-1}T_0^{-1},\ T_i\mapsto T_i,\ \pi_j\mapsto q^{-(\bar{ \varpi}_j, \bar{ \varpi}_j)/2}X_j\pi_j,\ t\mapsto t,\ q\mapsto q,
\end{align*}
for all $1\leqslant i\leqslant \ell$, $j\in J$.
Note that the automorphisms in the image of the subgroup $PSL^c(2, \bbZ)$ are $\bbC_{q,t}$-linear.
An $\bfH_{q,t}$-representation $V$ is called \emph{$PSL^c(2, \bbZ)$-invariant} if it carries a projective action of $PSL^c(2, \bbZ)$ such that 
$g\cdot H=g(H)\cdot g$ as elements in $\bbP End(V)$ for any $g\in PSL^c(2, \bbZ)$ and $H\in \bfH_{q,t}$. 

As in \cite[(3.2.2)]{Ch}, we will use the $\kappa$-function and write $t=q^\kappa$, and $\bfH_\kappa=\bfH_{q,t}$. Set $\bar\rho_\kappa=\kappa\bar\rho$.
The algebra $\bfH_\kappa$ has a polynomial representation $\calP{ol}=\bbC_{q,t}[X^{\pm 1}_1,...,X^{\pm 1}_\ell]$, see \cite[Thm.~3.2.1]{Ch} for its definition. Let $\ast$ be the involution on $\calP{ol}$ given by $X_b^\ast=X_{-b}$, $t^\ast=t^{-1}$, $q^\ast=q^{-1}$. For $f\in\calP{ol}$, we denote by $\langle f\rangle$ its constant term. The representation $\calP{ol}$ carries a nondegenerate $\ast$-bilinear form given by 
\[\langle f,g\rangle_\circ=\langle fg^\ast\mu_\circ\rangle,\]
where $\mu_\circ=\mu/\langle \mu\rangle$ with
\begin{align*}
\mu&=\prod_{\alpha\in\bar{\Phi}^\v_+}\prod_{i=0}^\infty\frac{(1-X_\alpha q_\alpha^i)(1-X_\alpha^{-1} q_\alpha^{i+1})}{(1-X_\alpha tq_\alpha^i)(1-X_\alpha^{-1} tq_\alpha^{i+1})},\\
\langle\mu\rangle&=\prod_{\alpha\in\bar{\Phi}^\v_+}\prod_{i=0}^\infty\frac{(1-q^{({\bar \rho}_\kappa,\alpha)+i)})^2}{(1-tq^{({\bar \rho}_\kappa,\alpha)+i)})(1-t^{-1}q^{({\bar \rho}_\kappa,\alpha)+i)})},
\end{align*}
where $q_\alpha=q^{(\alpha,\alpha)/2}$, see \cite[(3.3.1), (3.3.2)]{Ch}.

Nonsymmetric Macdonald polynomials $E_b$ indexed by $b\in P^\v$ are elements in $\calP{ol}':=\calP{ol}\otimes_{\bbC_{q,t}}\bbC(q^{\pm 1/m},t)$ subject to the conditions
\[E_b-X_b\in \bigoplus_{b'\succ b}\bbC(q^{\pm 1/m},t) X_{b'},\quad \langle E_b, X_{b'}\rangle_\circ =0\text{ for }b'\succ b.\]
They form an eigenbasis for the action of $Y$-operators on $\calP{ol}'$. More precisely, for any $f\in\bbC[X^{\pm 1}_1,...,X^{\pm 1}_\ell]$, let $L_f=f(Y_1,...,Y_\ell)$ be the corresponding operator in $\bfH_\kappa$. Then
\[L_f(E_b)=f(q^{-b_\sharp})E_b,\quad \text{ where }b_\sharp:=b-u_b^{-1}({\bar \rho}_\kappa) \text{ and } X_a(q^b)=q^{(a,b)},\ \forall\ a,\ b\in P^\v.\]
The renormalised Macdonald polynomials are defined as $\calE_b=E_b/E_b(q^{-{\bar \rho}_\kappa})$ for $b\in P^\v$.

\subsection{The finite dimensional representation $\bfV$}
From now on, we fix 
$$\kappa=-\frac{u}{h^\v}.$$ 
We view $q$ as a complex number, $t=q^\kappa$ and view $\bfH_\kappa$ as an algebra over $\bbC$.

In \cite[Thm.~3.10.6 (ii)]{Ch}, Cherednik showed that $\bfH_\kappa$ has a perfect representation $\frakF'[-{\bar \rho}_\kappa]$. As a vector space, it is the space of complex-valued functions on $\Sigma_u/\!\!\sim$, where $\Sigma_u$ is the set defined in \eqref{eq:sigmau}.
Indeed, since $(\theta^\v, {\bar \rho}_\kappa)+\kappa=-u$, the equivalent classes of weights in our $\Sigma_u$ are precisely those defined in \cite[(3.10.26)]{Ch}. 
Let $\chi_b$ be the characteristic function given by $\chi_b(\pi_{b'})=\delta_{b,b'}$ for $b,b' \in \Sigma_u/\!\!\sim$. As $b$ runs over $\Sigma_u/\!\!\sim$, they provide a basis of $\frakF'[-{\bar \rho}_\kappa]$.
Note that by Lemma \ref{lem:sigmau}, we have $\eW_u/\Omega_u=\Sigma_u/\!\!\sim$.
So we may also view $\frakF'[-{\bar \rho}_\kappa]$ as the space of complex-valued functions on $\eW_u/\Omega_u$.

On the other hand, $\frakF'[-{\bar \rho}_\kappa]$ is the unique simple quotient of the polynomial representation by \cite[Thm.~3.10.2]{Ch}. The form $\langle\ ,\ \rangle_\circ$ descends to $\frakF'[-{\bar \rho}_\kappa]$ and becomes 
\[\langle f, g\rangle_\bullet =\sum_{b\in\Sigma_u/\sim}f(\pi_b)g(\pi_b)^\ast\mu_\bullet(\pi_b),\]
where
\begin{equation}\label{eq:mubullet}
\mu_\bullet(\pi_b)=\prod_{\alpha+nc\in\Phi^\v(\pi_b)}\frac{t_\alpha^{-1/2}-q_\alpha^nt_\alpha^{1/2}X_\alpha(q^{-\bar\rho_\kappa})}{t_\alpha^{1/2}-q_\alpha^nt_\alpha^{-1/2}X_\alpha(q^{-\bar\rho_\kappa})},
\end{equation}
as in \cite[(3.4.6)]{Ch}, with $q_\alpha=q^{(\alpha,\alpha)/2}$, $t_\alpha=q_\alpha^\kappa$, $X_\alpha(q^{-\bar\rho_\kappa})=q^{(\alpha, -\bar\rho_\kappa)}$. The product runs over $\alpha\in\bar\Phi^\v$ and $n\in\bbZ$ such that $\alpha+nc\in\Phi^\v(\pi_b)$.

Let us describe the projective $PSL^c(2, \bbZ)$-action on $\frakF'[-{\bar \rho}_\kappa]$. We will consider a right action which is given by the transpose of the left $PSL^c(2, \bbZ)$-action defined in \cite{Ch}. In other words, the one given by $f\cdot g=g^{-1}\cdot f$ for $g\in PSL^c(2, \bbZ)$ and $f\in \frakF'[-{\bar \rho}_\kappa]$.
It is enough to describe the action of $\sigma$ and $\tau_+$. Note that, our $\sigma$ as in \eqref{eq:sigmatau} is the inverse of the $\sigma$ in \cite[(3.2.14)]{Ch}.

\begin{itemize}
\item \emph{The action of $\sigma$:} The Fourier transform defined in \cite[(3.4.12)]{Ch} is the map 
\begin{equation*}
\varphi_\circ: \calP{ol}\to \frakF'[-{\bar \rho}_\kappa],\quad f\mapsto (\pi_b\mapsto \langle f \calE_b\mu_0\rangle).
\end{equation*}
It factorises through the quotient map $\calP{ol}\to \frakF'[-{\bar \rho}_\kappa]$ and yields an isomorphism
\begin{eqnarray*}
\varphi_\bullet: \frakF'[-{\bar \rho}_\kappa]&\simto& \frakF'[-{\bar \rho}_\kappa],\\
 f&\mapsto& \Big(\pi_b\mapsto \sum_{b'\,\in\Sigma_u/\sim}f(\pi_{b'})\calE_b(\pi_{b'})\mu_\bullet(\pi_{b'})\Big),
\end{eqnarray*}
This map intertwines with the action of $\sigma$ on $\bfH_\kappa$.
In other words, for any $h\in\bfH_\kappa$, we have $\varphi_\bullet(hf)=\sigma(h)\varphi_\bullet(f)$.
So $\varphi_\bullet$ provides the desired $\sigma$-action on $\frakF'[-{\bar \rho}_\kappa]$.
Let $\bfS$ be the corresponding matrix on the basis $\{\chi_b\}$. The matrix coefficients are
$$\bfS_{b,b'}=\calE_{b}(q^{b'_\sharp})\mu_\bullet(q^{b'_\sharp}).$$

\item \emph{The action of $\tau_+$:} 
The action of $\tau_+$ is 
given by multiplication with the inverse of the restricted Gaussian $\gamma_\ast\in \frakF'[-{\bar \rho}_\kappa]$ given by $\gamma_\ast(\pi_b)=q^{\frac{1}{2}|b-u_b^{-1}({\bar \rho}_\kappa)|^2}$, see \cite[Prop.~3.10.1]{Ch}. 
Let $\bfT$ be the matrix for $\tau_+$ on the basis $\{\chi_b\}$. Then the matrix coefficients are
$$\bfT_{b,b'}=q^{-\frac{1}{2}|b-\kappa u_b^{-1}{\bar \rho}|^2}\delta_{b,b'}.$$
\end{itemize}

\subsection{DAHA interpretation for modularity of characters}

Now, we consider the specialisation of the perfect representation $\frakF'[-{\bar \rho}_\kappa]$
at $q=\zeta=e^{-\frac{2\pi i h^\v}{u}}$. 
This is a \emph{good cyclotomic} reduction in the sense of \cite[Section.~3.11.1]{Ch}.
Indeed, recall that $m=1$ for type $B_{2n}$ and $C_n$, $m=2$ for type $D_{2n}$, and $m=|\Omega|$ in the other cases, see e.g. \cite[Section~3.1]{Ch}. We see that $m$ always divides $h^\v$. So $q$ and $q^{1/m}$ are both primitive $u$-th root of unity. Next, since $u$ and $h^\v$ are coprime, the second condition in loc.~cit. is also satisfied with $\tilde{k}=0$ and $N=u$ in the notation there. 

Let $\bfV$ be the specialisation of $\frakF'[-{\bar \rho}_\kappa]$ at $q=\zeta$.
By \cite[Thm.~3.11.1]{Ch}, this representation is well defined and remains perfect.
Note that in this case $t=1$ and the specialisation of $\bfH_\kappa$ becomes the double Weyl group.

\begin{thm}\label{thm:modularity}
The isomorphism of vector spaces
\begin{equation}\label{psi}
a: \bbV\simto \bfV,\quad \ch_b\mapsto\epsilon(u_b)\chi_b
\end{equation}
intertwines the projective $PSL^c(2,\bbZ)$-actions on both sides.
\end{thm}

\begin{proof}
First, note that 
\begin{equation*}
\bfT|_{q=\zeta} = e^{\frac{\pi i|{\bar \rho}|^2}{2h\!^\v}}\bbT.
\end{equation*}
Next, by \eqref{eq:mubullet} we see that in this specialization $\mu_\bullet|_{q=\zeta}=1$.
We have $E_b(q^{-{\bar \rho}_\kappa})|_{q=\zeta}=\zeta^{({\bar \rho}_\kappa,b_-)}$ by \cite[(3.3.16), (3.3.17)]{Ch}, and
\begin{equation*}
\calE_b(X)|_{q=\zeta}=\zeta^{-({\bar \rho}_\kappa,b_-)}X_b.
\end{equation*}
Therefore
\begin{align*}
\bfS_{b,b'}|_{q=\zeta}&=\calE_b(q^{b'_\sharp})|_{q=\zeta}=\zeta^{-({\bar \rho}_\kappa,b_-)+(b'_\sharp, b)}
\end{align*}
Recall that $b_-=u_b(b)$ and $b'_\sharp=b'-u_{b'}^{-1}({\bar \rho}_\kappa)$. So 
$$-({\bar \rho}_\kappa,b_-)+(b'_\sharp, b)=(-u_b^{-1}({\bar \rho}_\kappa)-u_{b'}^{-1}({\bar \rho}_\kappa),b)+(b',b).$$
Note that $-u_b^{-1}({\bar \rho})-u_{b'}^{-1}({\bar \rho})\in Q^\v$ with $b\in P^\v$.
We are in the case $r=1$, so $Q^\v\subset Q$ and $(Q^\v,P^\v)\subset\bbZ$.
We deduce that $(-u_b^{-1}({\bar \rho}_\kappa)-u_{b'}^{-1}({\bar \rho}_\kappa),b)\in\kappa\bbZ$. Thus $\zeta^{(-u_b^{-1}({\bar \rho}_\kappa)-u_{b'}^{-1}({\bar \rho}_\kappa),b)}=1$. Therefore 
$$\bfS_{b,b'}|_{q=\zeta}=\zeta^{(b',b)}.$$
Then by \eqref{eq:ST}, we obtain
$$\bbS_{b,b'}=a\epsilon(u_bu_{b'})\bfS_{b,b'}|_{q=\zeta},\quad\text{ where }
a=\left|\frac{P^\v}{uh^\v  Q^\v}\right|^{-1/2}\left(\prod_{\alpha\in\bar\Phi^\v_+}2\sin\frac{\pi u(\alpha,{\bar \rho})}{h^\v}\right).$$
By consequence, the isomorphism $\bbP\End(\bbV)\simeq \bbP\End(\bfV)$ induced by $\bbV\simto \bfV$, $\ch_b\mapsto\epsilon(u_b)\chi_b$, sends $\bbS$, $\bbT$ to $\bfS$, $\bfT$ respectively.
\end{proof}

\begin{cor}\label{corSL2}
The projective action of $PSL^c(2,\bbZ)$ on $\bfV$ can be lifted to an $SL(2,\bbZ)$-action.
\end{cor}

\begin{rk}
A standard computation shows that $a^2=u^l$.
\end{rk}

\medskip

\section{$\calW$-algebras}
As in the previous section, we assume $r=1$.

\subsection{Representations}
Let $\bar{\frakg}$ be a finite dimensional simple Lie algebra.
Let $f$ be a nilpotent element in $\bar{\frakg}$. 
We assume that $f$ admits a good even grading\footnote{This assumption can be weaken by considering Ramond twisted modules.}
$$\bar{\frakg}=\bigoplus_{j\in\bbZ}\bar{\frakg}_j.$$
Without loss of generality, we may assume $\bar{\frakh}\subset \bar{\frakg}_0$ and the root system is compatible with the grading, i.e., each root space is contained in some $\bar\frakg_j$.


We assume also that the nilpotent element $f$ is of \emph{standard Levi type}, i.e., there is a Levi subalgebra $\bar{\frakl}$ of $\bar{\frakg}$ such that $f$ is a regular nilpotent element in $\bar{\frakl}$.
The Levi subalgebra $\bar{\frakl}$ can be realised as the centraliser of
$$\bar{\frakh}^f=\{v\in\bar{\frakh}\mid [f,v]=0\}.$$
The root system for $\bar{\frakl}$ consists of finite roots of $\bar\frakg$ which vanishes on ${\bar{\frakh}}^f$. Let $\bar{\Phi}^\v_f$ be the corresponding set of coroots, and denote by $W_f$ the corresponding Weyl group.

Let $V^k(\bar\frakg)=U(\frakg)\otimes_{U(\bar\frakg[z]\oplus\bbC c\oplus \bbC d)}\bbC_k$ be the universal affine Vertex algebra of level $k$, where $\bar\frakg[z]\oplus\bbC d$ acts trivially on $\bbC_k$ and $c$ acts by the scalar $k$. Let $L_k(\bar\frakg)$ be its simple quotient.
Assume that $k$ is a boundary principal admissible level.
We consider the category of $L_k(\bar\frakg)$-modules in the 
category $\calO$ of the affine Lie algebra $\frakg$, denoted by $L_k(\bar\frakg)\mod$. By \cite{A1}, it is a semi-simple category with simple objects given by the simple highest weight modules $L(\lambda)$ with $\lambda\in\Adm_k$.

The quantized Drinfeld-Sokolov reduction functor for an arbitrary $f$ was introduced by Kac-Roan-Wakimoto \cite{KRW}.
There are two versions of reduction functors $\Psi^{\pm}_f$. The $W$-algebra is the vertex algebra obtained by $\calW_k(\bar\frakg,f)=\Psi^+\big(L_k(\bar\frakg)\big)$. 
Both reductions provide functors
\[\Psi^{\pm}_f:  L_k(\bar\frakg)\mod\to \calW_k(\bar\frakg,f)\mod.\]
The functor $\Psi^-_f$ is expected to satisfy the following properties.

\begin{conj}[Kac-Roan-Wakimoto, Arakawa]\label{conj}
Assume $k$ is a boundary principal admissible level and $f$ is regular in the Levi subalgebra $\bar\frakl$.
The following holds:
\begin{enumerate}
\item The functor $\Psi^-_f$ is exact. The image of a simple module is either simple or zero.
\item For $\lambda\in\Adm_k$, the module $\Psi^-_f(L(\lambda))$ is nonzero if and only if 
$$\Phi^\v_{\lambda,+}\subset \Phi^\v_+\backslash \bar{\Phi}^\v_f.$$
\item For $\lambda,\lambda'\in\Adm_k$, we have $\Psi^-_f(L(\lambda))\simeq \Psi^-_f(L(\lambda'))$ if and only if $\lambda=x.\lambda'$ for some $x\in W_f$.
\end{enumerate}
\end{conj}

For principal nilpotent elements, this was first conjectured by Frenkel-Kac-Wakimoto \cite[Conjecture $3.4_{-}$, Proposition 3.4]{FKW} and proved first by Arakawa \cite{A2}, and later Dhillon-Raskin \cite{DR} gave a different proof using localization techniques.
Kac-Roan-Wakimoto \cite[Conjectures $3.1A$, $3.1B$]{KRW} gave partial generalizations of Frenkel-Kac-Wakimoto's conjecture to $f$ of standard Levi type.
The current form of the conjecture is suggested to us by Arakawa.

Besides the principal nilpotent case, it has also been proven for exceptional pairs of $f$ and $k$ for any $\bar\frakg$ by Arakawa-van Ekeren \cite{AvE}. Also, for arbitrary nilpotent element $f$ in type $A$, it holds for ordinary modules \cite{A3}.

Let $\Irr(\calW_k(\bar\frakg,f))$ be the set of isomorphism classes of simple highest weight $\calW_k(\bar\frakg,f)$-modules in the essential image of the functor $\Psi^-_f$.
Based on Conjecture \ref{conj}, it can be parametrised as follows:
first, since $\Psi_f^-$ is exact, the module $H_{f}^0(L(\lambda))$ is nonzero if and only if its character is nontrivial. By \cite[Thm.~2.3(c)]{KW07}, this is the case if and only if 
$\Phi^\v_{\lambda,+}\subset \Phi^\v_+\backslash \bar{\Phi}^\v_f$.
Now recall that $\lambda\in\Adm_k$ has the form $\lambda=x.(k \varpi_0)$ for $x\in\eW$ such that $x(\Pi^\v_u)\subset \Phi^\v_+$, and $\Phi^\v_\lambda$ is generated by $x(\Pi^\v_u)$ as root system. Therefore $H_{f}^0(L(\lambda))\neq 0$ if and only if 
$$x(\Pi^\v_u)\subset \Phi^\v_+\backslash \bar{\Phi}^\v_f.$$
So Conjecture \ref{conj} implies that $\Irr(\calW_k(\bar\frakg,f))$ is parametrised by the set
\begin{equation}\label{eq:admf}
\Adm_k^f=\{ x.(k \varpi_0)\mid x(\Pi^\v_u)\subset \Phi^\v_+\backslash \bar{\Phi}^\v_f\}/\sim_f,
\end{equation}
where the equivalence relation $\sim_f$ is given by $x\sim_f x'$ if $x'\in W_f x$.

\subsection{Matching with fixed points in affine Spaltenstein varieties}
Let $P_f^\v\subset G^\v$ be the parahoric subgroup containing $I^\v$ and such that the root system of its Levi factor is given by $\bar{\Phi}^\v_f$.
We consider the partial affine flag variety$\Fl^f=G^\v/P_f^\v$, and the affine Spaltenstein variety is defined by
\[\Fl_\gamma^f=\{g P_f^\v\mid \Ad_{g^{-1}}(\gamma)\in \rad(\Lie (P_f^\v))\}.\]
Here $\rad(\Lie (P_f^\v))$ refers to the radical of the Lie algebra of $P_f^\v$.
Let $\Fl_\gamma^{f,\circ}$ be the neutral component.
The $\bbC^\times$ defined in \eqref{Cstar} acts on $\Fl_\gamma^f$ by left multiplication.
Notice that the natural projection $\pi: \Fl\to\Fl^f$ restricts to the quotient map $\eW\to \eW/W_f$ at the level of $T^\v$-fixed points.
Next, note that the roots in $\rad(\Lie (P_f^\v))$ are precisely $\Phi^\v_+\backslash \bar{\Phi}^\v_f$.
So for $y\in \eW$, we have
\begin{eqnarray*}
y\in \Fl^\gamma_f  &\Longleftrightarrow& \Ad_{y^{-1}}(\gamma)\in\rad(\Lie (P_f^\v))\\
&\Longleftrightarrow& y^{-1}(\Pi_u)\subset \Phi^\v_+\backslash \bar{\Phi}^\v_f.
\end{eqnarray*}

Now, Theorem \ref{thm:adm} and \eqref{eq:admf} implies a bijection between $(\Fl_\gamma^{f,\circ})^{\bbC^\times}$ and $\Adm_k^f$.

\begin{thm}\label{thm:Wbij}
Assuming Conjecture \ref{conj} holds, then the bijection \eqref{eq:bij}
induces a bijection 
\begin{equation}
(\Fl_\gamma^{f,\circ})^{\bbC^\times}\simeq \Irr(\calW_k(\bar\frakg,f)).
\end{equation}
\end{thm}

\medskip

\subsection{Counting of simple highest weight $\calW_k(\bar\frakg,f)$-modules}

The bijection above reduces the counting of the number of simple highest weight $\calW_k(\bar\frakg,f)$-modules to the counting of fixed points in $(\Fl_\gamma^{f,\circ})^{\bbC^\times}$. This can be done as below, following ideas in the work of Sommers \cite{S}.

Let us first give another description of elements in $\Adm^f_k$.
For any integer $n$, let $\Pi^\v_n=\{(n-1)c+\alpha_0^\v, \alpha^\v_i, i\in \bar I\}$.
The $n$-th fundamental alcove
$$\calA_n=\{ v\in \bar\frakh_\bbR\mid (\alpha^\v, v)\leqslant 0,\ \forall \alpha^\v\in\Pi^\v_n\}$$
is a fundamental domain for the action of $\bar W\ltimes n Q^\v$ on $\bar\frakh_\bbR$.
For $x\in \eW$, we have $x(\Pi^\v_u)\subset \Phi^\v_+$ if and only if $x(\calA_1)\subset \calA_u$. Thus $\eW_u=\{x\in\eW\mid x(\calA_1)\subset \calA_u\}$.
Since $\calA_u$ is a fundamental domain for the action of $\bar W\ltimes u Q^\v$. Elements in $\eW_u$ are in bijection with the right cosets of $\bar W\ltimes u Q^\v$ on $\eW$. The quotient $\bar W\ltimes u Q^\v\backslash \eW$ identifies naturally with $\calS_u:=P^\v/uQ^\v$ with the map given by
\[(\bar W\ltimes u Q^\v)\backslash\eW\simto \calS_u,\quad x\mapsto -x^{-1}(0).\]
Further, this isomorphism is $\eW$-equivariant with the action induced by the right multiplication of $\eW$ on itself on the left hand side, and the natural action of $\eW$ on $\calS_u$ on the right hand side. Note that on the set of representatives, this isomorphism is given by
\begin{equation}\label{isoSu}
\eW_u\simto \calS_u,\quad \pi_b\mapsto -\pi_b(0)=b_-.
\end{equation}

\begin{lemma}\label{lem:stab}
For $\pi_b\in\eW_u$, we have $\pi_b(\Pi^\v_u)\subset\Phi^\v_+\backslash\bar{\Phi}^\v_f$ if and only if 
the stabilizer of $b$  for the $W_f$-action on $\calS_u$ is trivial.
\end{lemma}
\begin{proof}
If $\pi_b(\Pi^\v_u)$ is not contained in $\Phi^\v_+\backslash \bar{\Phi}^\v_f$, then there exists $\alpha^\v\in\Pi^\v_u$ such that $\pi_b(\alpha^\v)=\beta^\v$ belongs to $\bar{\Phi}^\v_f$. By \eqref{eq:desc}, if $\alpha^\v=\alpha^\v_i$ for some $i$ in $\bar I$, then $(u_b^{-1}(\alpha^\v_i),b)=0$ and $u_b^{-1}(\alpha^\v_i)$ belongs to $\bar{\Phi}^\v_f$; if $\alpha^\v=u\delta-\theta^\v$, then $(-u_b^{-1}(\theta^\v),b)=u$ and $-u_b^{-1}(\theta^\v)$ belongs to $\bar{\Phi}^\v_f$.
In both cases, the reflection $s_{\beta^\v}$ stabilizes $b\in \calS_u$. Thus the stabilizer of $b$ in $W_f$ is nontrivial.

For the other direction, assume $b$ has a nontrivial stabilizer in $W_f$. 
Then by \cite[Prop.~4.1]{S} this stabilizer is a parabolic subgroup of $W_f$.
Thus there exists ${\beta^\v}\in\bar{\Phi}_f^\v$ such that $(\beta^\v, b)\cong 0$ mod $u$. Equivalently $(u_b(\beta^\v),b_-)\cong 0$ mod $u$. Note that $b_-\in\calA_u$. This implies either $u_b(\beta^\v)=\pm \alpha_i^\v$ for some $i\in\bar I$ and $(u_b(\beta^\v),b_-)=0$, or $u_b(\beta^\v)=\mp \theta^\v$ and $(u_b(\beta^\v),b_-)=\pm u$.
Again, by \eqref{eq:desc}, this means one of $\pm\beta^\v$ belong to $\pi_b(\Pi_u^\v)$. We conclude that $\pi_b(\Pi_u^\v)\cap \bar\Phi^\v_f$ is nonempty.
\end{proof}

We deduce the following formula for the cardinality of $(\Fl_\gamma^{f,\circ})^{\bbC^\times}$. Recall that $\ell$ is the rank of $\bar\frakg$. Let $m_1,...,m_j$ be the exponents of the Weyl group $W_f$.

\begin{prop}\label{prop:counting}
We have
$$|(\Fl_\gamma^{f,\circ})^{\bbC^\times}|=\frac{1}{|W_f|}u^{\ell-j}(u-m_1)...(u-m_j).$$
\end{prop}
\begin{proof}
We follow a similar argument as in the proof of \cite[Prop.~6.1]{S}.
The question is equivalent to calculate $|\Adm_k^f|$.
Let 
\[\eW_{u,f}=\{x\in\eW_u\mid x(\Pi^\v_u)\subset\Phi^\v_+\backslash\bar{\Phi}^\v_f\}.\]
It is a subset of $\eW_u$, which is stable under the action of $\Omega_u$.
By \eqref{eq:admf}, there is a bijection between $\Adm_k^f$ and $W_f\backslash \eW_{u,f}/\Omega_u$.
The lemma above implies that the set $W_f\backslash \eW_{u,f}$ is in bijection with free $W_f$-orbits in $\calS_u$. Write $e$ for the cardinality of $\Omega$.
By \cite[Prop.~3.9]{S}, for any $w\in W_f$, the number of $w$-fixed points in $\calS_u$ is $eu^{d(w)}$, where $d(w)$ is the dimension of $\frakh^w$. We claim that 
\[|\{v\in \calS_u\mid Z_{W_f}(v)=1\}|=\sum_{w\in W_f}\epsilon(w)|(\calS_u)^w|.\]
Indeed, consider the action of $W_f$ on $\calS_u$,  let $X=\{(w,h)\in W_f\times\calS_u\mid wh=h\}$. Counting the function $\sum_{(w,h)\in X}\epsilon(w)$ in two ways give the equality
\begin{equation}\label{eq: counting}
\sum_{w\in W_f}\epsilon(w)|(\calS_u)^w|=\sum_{v\in \calS_u}\sum_{w\in Z_{W_f}(v)}\epsilon(w).
\end{equation}
Now, by \cite[Prop.~4.1]{S}, for any $v\in\calS_u$ the stabilizer $Z_{W_f}(v)$ is always a parabolic subgroup of $W_f$. In particular, it contains an element of length one. Therefore $\sum_{w\in Z_{W_f}(v)}\epsilon(w)=0$ whenever $Z_{W_f}(v)$ is nontrivial. So the right hand side of \eqref{eq: counting} equals $|\{v\in \calS_u\mid Z_{W_f}(v)=1\}|$. The claim follows. 

We deduce that
\begin{align*}
|W_f\backslash \eW_{u,f}|&=\frac{1}{|W_f|}|\{v\in \calS_u\mid Z_{W_f}(v)=1\}|\\
&=\frac{1}{|W_f|}|\sum_{w\in W_f}\epsilon(w)eu^{d(w)}\\
&=\frac{1}{|W_f|}|\sum_{w\in W_J}(-1)^{s(w)}eu^{\ell-s(w)}\\
&=\frac{1}{|W_f|}eu^{\ell-j}(u-m_1)...(u-m_j).
\end{align*}
Here the third equality is given by the fact $d(w)=\ell-s(w)$, where $s(w)$ is the least number of reflections needed to write $w$ as a product of reflections. So $l(w)$ and $s(w)$ have the same parity. The last equality follows from the Shepard-Todd theorem.

Finally, since $\Omega_u$ acts on $W_f\backslash \eW_{u,f}$ freely, we deduce 
\[|\Adm^f_k|=|W_f\backslash \eW_{u,f}/\Omega_u|=\frac{1}{|W_f|}u^{l-j}(u-m_1)...(u-m_j).\]
\end{proof}

\smallskip

\begin{cor}\label{cor:count}
Assuming Conjecture \ref{conj} holds, then
\begin{equation}\label{count}
|\Irr(\calW_k(\bar\frakg,f))|=\frac{1}{|W_f|}u^{\ell-j}(u-m_1)...(u-m_j).
\end{equation}
In particular, $\calW_k(\bar\frakg,f)=0$ if and only if $u$ is equal to one of the exponents of $W_f$.
\end{cor}
\begin{proof}
The first statement follows from Theorem \ref{thm:Wbij} and the previous proposition. 
The second is because the $\calW_k(\bar\frakg,f)$ has no simple highest weight module if and only if it is zero.
\end{proof}

\begin{rk}\label{rk1}
This gives a new way to detect when the $\calW$-algebra $\calW_k(\bar\frakg,f)$ is zero.
It is known that $\calW_k(\bar\frakg,f)=0$ if and only if $f$ does not belong to the associated variety of $L_k(\bar\frakg)$.
The associated variety of $L_k(\bar\frakg)$ is the closure of some nilpotent orbit $O_k$, which was explicitly determined by Arakawa \cite{A1}.
We checked  by a case by case computation that when $f$ is of standard Levi type, $f\not\in\bar{O}_k$ is indeed equivalent to $u$ being equal to one of the exponents of $W_f$.
\end{rk}

\begin{rk}\label{rk2}
The formula \eqref{count} also gives another way to determine when $\calW_k(\bar\frakg,f)=\bbC$.
We checked that the table \ref{table:trivialVOA} below gives a complete list for which $|\Irr(\calW_k(\bar\frakg,f))|=1$ according to \eqref{count}. These are precisely the cases when $\calW_k(\bar\frakg,f)=\bbC$ for boundary principal admissible level and $f$ regular in a Levi. This matches with the results in \cite{AvEM}, which are obtained by different methods.
\begin{table}[h]
\begin{tabular}{c|c|c}
$\bar{\frakg}$ & $u$ & $f$ \\ \hline
$\bar{\frakg}$ & $h+1$ & principal \\ \hline
$\sl_{ul+1}$ & any & $[\underbrace{u,\cdots,u}_{l},1]$ \\ \hline
$\so_{ul+1}$ & odd & $[\underbrace{u,\cdots,u}_{l},1]$, $l$ even \\ \hline
$\sp_{ul}$  & odd & $[\underbrace{u,\cdots, u}_l]$, $l$ even \\ \hline
$\so_{ul}$ & odd & $[\underbrace{u,\cdots, u}_l]$, $l$ even \\ \hline
$\frake_7$ & $7$ & $A_6$ 
\end{tabular}
\caption{\label{table:trivialVOA} the case $|\Irr(\calW_k(\bar\frakg,f))|=1$, with $h$ being the Coxeter number of $\bar W$.}
\end{table}
\end{rk}

\medskip

\subsection{Modularity}

Finally, let us discuss how the $SL_2(\bbZ)$-actions are compatible with reductions.

Recall that for DAHA, we specialised the parameter $t=1$. The algebra $\bfH_\kappa$ contains a copy of the group algebra of $\bar{W}$. Let $\bfe_f=\sum_{w\in W_f}\epsilon(w)w$ be the anti-symmetrizer.
The subspace $\bfV_f:=\bfe_f\bfV$ in $\bfV$ is a module over the subalgebra $\bfe_f\bfH_\kappa\bfe_f$, and is stable under the $SL_2(\bbZ)$-action.
As a $\bar W$-representation $\bfV$ is identified with functions on $\calS_u/\Omega_u$ by \eqref{isoSu}. So it is a direct sum of induced representations $\Ind_{P_b}^{\bar W}(1)$ with $b$ varies in $\calS_u$. Here $P_b$ is the parabolic subgroup in $\bar W$ given by the stabiliser of $b$.
By Mackey formula, the factor $\bfe_f(\Ind_{P_b}^{\bar W}(1))=0$ if and only if $P_b\cap W_f$ is non-empty. So $\bfV_f$ has a basis $\bfe_f\chi_b$ where $b$ runs over a set of representatives of free $W_f$-orbit in $\calS_u/\Omega_u$.

Since the category of admissible representations is semi-simple, we can identify $\bbV$ with its Grothendieck group. Let $\bbV_f$ be the space spanned by the characters of simple highest weight $\calW_k(\bar\frakg,f)$-modules. Then it is identified with the Grothendieck group of the Serre subcategory in $\calW_k(\bar\frakg,f)\mod$ generated by $\Irr(\calW_k(\bar\frakg,f))$. The reduction functor $\Psi^-_f$ is exact, hence it induces a map on the Grothendieck group $\bbV\to \bbV_f$.

It follows from Lemma \ref{lem:stab} and the discussion above that for $\lambda=\pi_b.(k \varpi_0)$, the reduction $\Psi^-_f(L(\lambda))$ is zero if and only if $\bfe_f\chi_b=0$; and $\Psi^-_f(L(\lambda))=\Psi^-_f(L(\lambda'))$ if and only if $\pi_{b'}=w\pi_b$ for some $w\in W_f$, and in this case $\bfe_f\chi_{b'}=\epsilon(w) \bfe_f\chi_b$.
By consequence, the isomorphism $a:\bbV\simto\bfV$ descends to a unique isomorphism $a^f: \bbV_f\to \bfV_f$ such that the following diagram commute
\[\xymatrix{ \bbV\ar[r]^{a}\ar[d]_{\Psi^-_f} &\bfV \ar[d]^{\bfe_f} \\
\bbV_f\ar[r]^{a^f} &\bfV_f.
}\]
The action of $SL_2(\bbZ)$ commute with the Drinfeld-Sokolov reduction on the representation side, and commute with $\bfe_f$ on the DAHA side. Therefore, we deduce the following result.

\begin{thm}\label{thm:modW}
The following isomorphism
\[a_f: \bbV_f\to \bfV_f,\quad \Psi^-_f(\ch_b)\mapsto\epsilon(u_b)\bfe_f\chi_b\]
intertwines the $SL_2(\bbZ)$-action on both sides.
\end{thm}

\begin{rk}
When $\bar{\frakg}=A_1$ and $f$ is an element in the principal orbit, the above result was first observed in \cite{KSY}. The generalisation to $\bar{\frakg}=A_{N-1}$ and $f$ principal was first presented in \cite{GKNPS}.
\end{rk}


\end{document}